\documentclass[12pt]{amsart}
\usepackage{amsmath}
\usepackage{graphicx}
\usepackage{amsfonts}
\usepackage{amstext}
\usepackage{amsbsy}
\usepackage{amsopn}
\usepackage{amsxtra}
\usepackage{upref}
\usepackage{amsthm}
\usepackage{amsmath}
\usepackage{amssymb}

\newtheorem{thm}{Theorem}[section]
\newtheorem{cor}[thm]{Corollary}
\newtheorem{lem}[thm]{Lemma}

\theoremstyle{definition}
\newtheorem{defn}[thm]{Definition}
\theoremstyle{remark}

\numberwithin{equation}{section}

\newcommand{\noi}{\noindent}
\newcommand{\sm}{\smallskip}
\newcommand{\me}{\medskip}
\newcommand{\bi}{\bigskip}

\newcommand{\ord}{ord \,{\mathcal F}_{\alpha}}
\def\Fs{{\mathcal F}}

\def\C{{\mathbb C}^n}
\def\B{{\mathbb B}^n}
\def\D{{\mathbb D}}

\def\S{{ S^{k}_{ij}}}
\def\SS{{\mathcal{S}}}
\def\la{{\lambda}}
\def\La{{\Lambda}}
\def\Fa{{\mathcal F}_{\alpha}}

\def\Aa{\mathcal A_\alpha}
\begin{document}

\title {The Order of a Linearly Invariant Family in $\C$}

\author{Martin Chuaqui and Rodrigo Hern\'andez}
\thanks{The
authors were partially supported by Fondecyt Grants  \#1110321 and \#1110160.
\endgraf  {\sl Key words:} Schwarzian derivative, homeomorphic extension, ball, univalence,
convexity,
Bergman metric, weakly linearly convex, projective dual space.
\endgraf {\sl 2000 AMS Subject Classification}. Primary: 32H02, 32A17;\,
Secondary: 30C45.}
%

\begin{abstract} We study the (trace) norm of a linearly invariant family in the ball
in $\C$. By adapting an approach that in one variable yields optimal results, we are
able to derive an upper bound for the norm of the family in terms of the Schwarzian
norm and the dimension $n$.
\end{abstract}
\maketitle

\section{Introduction}
The purpose of this paper is to obtain an upper bound for the trace order of a certain linearly invariant
family of locally biholomorphic mappings defined in the unit ball $\B$ in $\C$. The family is defined
in terms of the Schwarzian derivative $\SS F$, which inherits from the Bergman metric in $\B$ a natural norm $||\SS F||$
that is invariant under the automorphism group  \cite{RH1}. Disregarding certain
normalizations, the families $\Fa$ considered in this paper are defined by the condition $||\SS F||\leq \alpha$.
Linearly invariant families of holomorphic mappings were introduced in one complex variable by Pommerenke
in two seminal papers that offered a systematic treatment of such families \cite{Pom64}. He showed that relevant aspects of the family $\Fs$,
such as growth and covering, are determined by
its order $\sup_{f\in\Fs} |a_2(f)|$. If $Sf$ is the usual Schwarzian derivative
and $||Sf||=\sup_{|z|<1}(1-|z|^2)^2|Sf(z)|$, then the family of properly normalized locally univalent mappings
in the disc $\D$ for which $||Sf||\leq \alpha$ is linearly invariant. By means of a variational method, Pommerenke determined the sharp value $\sqrt{1+\frac12\alpha}\;$ for its order. In several variables, the concept of order of a linearly invariant family appears in the
form of the (trace) order and the (norm) order \cite{GK}.
In this work, we mimic the variational approach in several variables to estimate the order of $\Fa$ in terms of $\alpha$
and the dimension $n$. Much like in the analysis found in \cite{Pom64}, we are lead to a characteristic equation involving derivatives
of order up to three that must be satisfied by any mapping extremal for the trace order. Finally, the estimate on the trace order is used to obtain a similar estimate for the norm order of the family $\Fa$.

\section{Preliminaries}

In \cite{O} T.Oda generalizes the concept of Schwarzian derivative to the case of locally biholomorphic
mappings in several variables. For such a mapping $F=F=(f_1,\ldots,f_n):\Omega\subset\C\rightarrow\C$ he
introduces a family of Schwarzian derivatives by
\begin{equation}\label{schwarzian-def}\S F= \displaystyle
\sum^{n}_{l=1}\frac{\partial^{2} f_{l}}{\partial z_{i}\partial
z_{j}} \frac{\partial z_{k}}{\partial
f_{l}}-\frac{1}{n+1}\left(\delta^{k}_{i}\frac{\partial}{\partial
z_{j}}+\delta^{k}_{j}\frac{\partial}{\partial z_{i}}\right)
\log\,JF\, ,\end{equation} where $i,j,k=1,2, \ldots,n,$
$JF=\det(DF)$ is the jacobian determinant of the diferential $DF$
and $\delta^{k}_{i}$ are the Kronecker symbols. Two important aspects of the
one dimensional Schwarzian are also present in this context. First,
\begin{equation}\label{moebius-property}\S F=0 \;\;\; \mbox{for
all}\quad i,j,k=1,2,\ldots ,n
 \;\;\;\mbox{iff}\;\;\;F(z)= M(z)\, ,\end{equation} for some
M\"{o}bius transformation
$$M(z)=\left(\frac{l_1(z)}{l_0(z)},\ldots,\frac{l_n(z)}{l_0(z)}\right)\, ,$$
where $l_i(z)=a_{i0}+a_{i1}z_1+\cdots+a_{in}z_n$ with
$\det(a_{ij})\neq 0$. Next, under composition we have the chain rule

\begin{equation}\label{cadena}\S(G\circ F)(z)=\S F(z)+ \sum^{n}_{l,m,r=1}\mathbb
S^{r}_{lm}G(w)\frac{\partial w_{l}} {\partial z_{i}}\frac{\partial
w_{m}} {\partial z_{j}} \frac{\partial z_{k}} {\partial w_{r}}\; ,
\; w=F(z) \, .\end{equation} Thus, if $G$ is a M\"{o}bius
transformation then $\S(G\circ F)= \S F.$ The $S^0_{ij}F$
coefficients are given by
$$S^0_{ij}F(z)=(JF)^{-\frac{1}{n+1}}\left(\frac{\partial^2}{\partial z_i\partial z_j}
(JF)^{-\frac{1}{n+1}}-\sum_{k=1}^n\,\frac{\partial}{\partial
z_k}(JF)^{-\frac{1}{n+1}}S^k_{ij}F(z)\right)\, .$$

One can find in the literature other equivalent formulations of the Schwarzian in several variables,
which also come in the form of differential operators of orders two and three (see, {\it e.g.}, \cite{MP1}, \cite{MP2}, \cite{MT}).
In order to recover a mapping from its Schwarzian derivatives we can consider the following overdetermined system
of partial differential equations,
\begin{equation}\label{sistema}
\frac{\partial^{2}u}{\partial z_{i}\partial z_{j}} =
\sum^{n}_{k=1}P^{k}_{ij}(z)\frac{\partial u}{\partial z_{k}}+
P^{0}_{ij}(z)u \; , \quad i,j= 1,2, \ldots,n\, ,\end{equation} where
$z=(z_{1},z_{2},...,z_n)\in \Omega$ and $P^k_{ij}(z)$ are
holomorphic functions in $\Omega$, for $i,j,k=0,\ldots,n$. The system (\ref{sistema})  is called
{\it completely integrable} if there are\,  $n+1$ (maximun)
linearly independent solutions. The system is said to be in {\it canonical
form} (see \cite{Y76}) if the coefficients satisfy
$$\sum_{j=1}^{n} P_{ij}^{j}(z)=0 \; ,\quad i=1,2,\ldots,n\, .$$
An important result established by Oda is that (\ref{sistema}) is completely integrable and in canonical form
if and only if $P^k_{ij}=\S F$ for a locally boholomorphic mapping $F=(f_1, \ldots, f_n)$,
where $f_i=u_i/u_0$ for $1\leq i\leq n$ and $u_0, u_1,\ldots, u_n$ is a set of linearly independent
solutions of the system. It was also observed by the author that $u_0=\left(JF\right)^{-\frac{1}{n+1}}$ is always
a solution of (\ref{sistema}) with $P_{ij}^k=\S F$. The following result not stated in
the work of Oda will be important in the rest of the paper.

\bi
The individual components $\S F$ can be gathered to conform an operator in the following form (see\cite{RH1}).

\begin{defn}
For $k=1,\ldots,n$  let $\mathbb S^{k}F$ be the
matrix
$$ \mathbb S^{k} F= (\S F)\, ,\quad i,j=1,\ldots,n\, .$$
\end{defn}

\begin{defn} We define the \textit{Schwarzian derivative operator} as the
mapping $\SS F(z):T_z\Omega \to T_{F(z)}F(\Omega)$ given by

$$\SS F(z)(\vec{v})=\left(\, \vec{v}^{\,t}\mathbb S^{1}F(z)\vec{v}\, ,
\, \vec{v}^{\,t}\mathbb S^{2}F(z)\vec{v} \, ,\,
\ldots,\vec{v}^{\,t}\mathbb S^{n}F(z)\vec{v}\, \right)\, ,$$

\sm
\noi
where $\vec{v}\in T_z \Omega$.

\end{defn}

\me

As an  operator $\SS F(z)$ inherits a norm from the metric in the domain:

\begin{equation} \|\SS F(z)\|=\displaystyle\sup_{\|\vec v\|=1}\|\SS F(z)(\vec v\,)\|\,
 ,\end{equation}
and finally, we let
\begin{equation} ||\SS F|| = \sup_{z\in\Omega}||\SS F(z)||\, . \end{equation}

Our interest is to study certain classes of locally biholomorphic mappings $F$
defined in the unit ball $\B$. The Bergman metric $g_B$ on $\B$ is the hermitian product defined by
\begin{equation}\label{bergman metric}
g_{ij}(z)=\frac{n+1}{(1-|z|^2)^2}\,\left[(1-|z|^2)\delta_{ij}+\bar{z}_iz_j\right].
\end{equation}
The automorphisms of $\B$ act as isometries of the Bergman metric, and are given by
$$\sigma(z)=\frac{Az+B}{Cz+D}\, ,$$ where $A$ is $n\times n$, $B$ is
$n\times 1$, $C$ is $1\times n$ and $D$ is $1\times 1$ with
$$\begin{array}{ccl}A^t\overline A-C^t\overline C&=& {\rm Id}\, ,\\
                    |D|^2-B^t\overline B&=&1\, ,\\
                    A^t\overline B-C^t\overline D&=&0\, ,\end{array}$$
(see, e.g., \cite{K}).

By appealing to the chain rule (\ref{cadena}), it was shown in \cite{RH1} that
$$ \|\SS (F\circ \sigma)(z) \|= \|\SS F(\sigma(z))\| \, ,$$
from which
\begin{equation}
\|\SS F\| = \|\SS (F\circ\sigma)\| \, . \end{equation}

In this paper we will consider the family  $\Fa$ defined by
$$ \Fa = \{ \, F:\B\rightarrow\C \; | \;\, F \;\, {\rm locally\; biholomorphic} \, , \, F(0)=0\, , \, DF(0)={\rm Id}\, , \; \|\SS F\|\leq \alpha\, \} \, . $$

\me
The family $\Fa$ is linearly invariant and also compact \cite{RH1}. We are interested in studying its (trace) order \cite{GK}, given by
\begin{equation}\label{orden}  \ord = \sup_{F\in\Fa}\,\sup_{|w|=1}\, \frac12\left|\,\sum_{i,j=1}^n\frac{\partial^2f_j}{\partial z_i\partial z_j}(0) w_i\right| \, .\end{equation}
Because the family is compact, the order is finite. An equivalent form of the order is given by
$$  \Aa =\sup_{f\in\Fa}|\nabla(JF)(0)|  \, ,$$
which  is  shown in \cite{RH1} to satisfy
$$\Aa =2 \ord \, .$$ A second measure of the size of a linearly invariant family $\Fs$ is given by the norm order, defined
by
$$||ord||\Fs=\sup_{f\in\mathcal{F}}\frac12||D^2F(0)|| \, ,$$ where
$$ F(z)=z+\frac12D^2F(0)(z,z)+\cdots \, . $$
In general, $ord\,\Fs \leq n||ord||\Fs $.
For the family $\Fa$ in particular, it was shown in \cite{RH1} that
\begin{equation}\label{ordenes} 1+\frac{\sqrt{3}}{2}\,\alpha \, \leq  \, ||ord||\Fa \, \leq \, \frac{2}{n+1}\,\ord +\frac{\sqrt{n+1}}{2}\,\alpha\, .
\end{equation}

\section{Variations and Extremal Mappings}

Let $F_0\in\Fa$ be a mapping for which $\Aa$ is maximal, with $\nabla(JF_0)(0)=\La=(\la_1,\ldots,\la_n)$. Let $\sigma$ be an automorphism of $\B$ with
$\sigma(0)=\zeta$, and consider the Koebe transform
$$G(z)=D\sigma(0)^{-1}DF_0(\zeta)^{-1}\left[F_0(\sigma(z))-F_0(\zeta)\right]\, .$$
\me\noi The mapping $G\in\Fa$ represents
a variation of the extremal mapping $F_0$ when $|\zeta|$ is small. With this in mind,
we need to compute $\nabla (JG)(0)$. We have that
$$ DG(z)=D\sigma(0)^{-1}DF_0(\zeta)^{-1}DF_0(\sigma(z))D\sigma(z) \, ,$$
hence
$$JG(z)=J\sigma(0)^{-1}JF_0(\zeta)^{-1}JF_0(\sigma(z))J\sigma(z) \, , $$
so that
\begin{equation}\label{variation}
\nabla(JG)(0)=\frac{\nabla(JG)}{JG}(0)=\frac{\nabla(JF_0)}{JF_0}(\zeta)D\sigma(0)+\frac{\nabla(J\sigma)}{J\sigma}(0) \, .
\end{equation}
In order to proceed with the analysis, we need the expansion of $\nabla(JG)(0)$ in powers of $\zeta$.

\begin{lem} Let
$$ B_{ij}=\sum_{k=1}^nS_{ij}^kF_0(0)\la_k \quad , \quad  B_{ij}^0=S_{ij}^0F_0(0)\,.$$ Then

$$\frac{\nabla(JF_0)}{JF_0}(\zeta)=\La+A\cdot\zeta+O(|\zeta|^2) \; \, , \;\, |\zeta|\rightarrow 0 \, ,$$
where $A=(A_{ij})$ is the matrix given by
\begin{equation}\label{matrix} A_{ij}=B_{ij}-(n+1)B_{ij}^0+\frac{\la_i\la_j}{n+1} \, .\end{equation}

\end{lem}

\begin{proof} Let $u_0=(JF_0)^{-\frac{1}{n+1}}$ and $\phi(\zeta)=\frac{\nabla(JF_0)}{JF_0}(\zeta)$. Then $\phi(0)=\La$ because $JF_0(0)=1$.
We have that $\phi=(\phi_1,\ldots,\phi_n)$, where
\sm
$$ \phi_i(\zeta)=\partial_i\log(JF_0)(\zeta)=-(n+1)\partial_i(\log u_0)(\zeta) \; , \quad \partial_i=\partial/\partial z_i \, .$$
\me\noi
Since $u_0$ is a solution of (\ref{sistema}) with $u_0(0)=1$ and $\nabla u_0(0)=-\frac{1}{n+1}\nabla (JF_0)(0)$, we see that
$$\partial_j\phi_i(0)=-(n+1)\partial_j\left[\frac{\partial_i u_0}{u_0}\right](0)=-(n+1)\left[\partial^2_{ij}u_0(0)-\partial_iu_0(0)\partial_ju_0(0)\right]
\, , $$$$=-(n+1)\left[-\frac{B_{ij}}{n+1}+B_{ij}^0-\frac{\la_i\la_j}{(n+1)^2}\right] \, ,$$
which  gives that the differential $D\phi(0)$ is given by the matrix $A=(A_{ij})$. This proves the lemma.
\end{proof}

\begin{lem} With the notation as before, one can choose $\sigma$ so that
$$ D\sigma(0)={\rm Id}+O(|\zeta|^2) \, ,$$
$$\frac{\nabla(J\sigma)}{J\sigma}(0)=-(n+1)\overline{\zeta} \, .$$
\end{lem}

\begin{proof}
Assume first that $\sigma(0)=\zeta=(\zeta_1,0,\ldots,0)$. Then we may take
$$\sigma(z)=\left(\frac{z_1+\zeta_1}{1+\overline\zeta_1 z_1},\frac{\sqrt{1-|\zeta_1|^2}z_2}{1+\overline\zeta_1 z_1},
\ldots, \frac{\sqrt{1-|\zeta_1|^2}z_n}{1+\overline\zeta_1 z_1}\right)\, ,$$ and one finds that
$$J\sigma(z)=\frac{(1-|\zeta_1|^2)^{\frac{n+1}{2}}}{(1+\overline\zeta_1
z_1)^{n+1}}\, ,$$ together with
$$D\sigma(0)=\left(
\begin{array}{ccccc}
(1-|\zeta_1|^2) & 0 & 0 & \cdots & 0 \\
0 & \sqrt{1-|\zeta_1|^2} & 0 & \ldots & 0\\
0 & 0 & \sqrt{1-|\zeta_1|^2} & \cdots & 0\\ \vdots & \vdots & \vdots & \vdots & \vdots\\
0 & 0 & 0 &\cdots & \sqrt{1-|\zeta_1|^2} \\
\end{array}
\right)\, ,$$
\sm\noi
from which the lemma follows for $\zeta$ of the form $(\zeta_1,0,\ldots,0)$. The general case obtains
after considering a rotation of the ball.
\end{proof}

In light of Lemmas 3.1 and 3.2, we can rewrite equation (3.1) as
\begin{equation}\label{variation-1}
\nabla(JG)(0)=\La+A\cdot\zeta-(n+1)\overline{\zeta}+O(|\zeta|^2) \, .
\end{equation}

\begin{thm} Let $F_0\in \Fa$ be extremal for the order, with $\nabla (JF_0)(0)=\La$. Then
\begin{equation}\label{relation} A\cdot\overline{\La}=(n+1)\overline{\La} \, .\end{equation}
\end{thm}

\begin{proof} The proof is based on the observation that, in reference to equation (\ref{variation-1}), we must have
$$|\nabla(JG)(0)| \leq |\La| \;\; , \;\, |\zeta|\rightarrow 0 \, . $$

\me\noi
Let $\langle v,w\rangle =v_1\overline{w_1}+\cdots+v_n\overline{w_n}$. Then
\sm
$$|\nabla(JG)(0)|^2=|\La|^2+{\rm Re}\langle\La, A\cdot\zeta\rangle-2(n+1){\rm Re}\langle\La,\zeta\rangle+O(|\zeta|^2)$$$$
\hspace{1,46cm}=|\La|^2+2{\rm Re}\langle\overline{A^t}\cdot\La-(n+1)\La,\zeta\rangle+O(|\zeta|^2) \, .$$

\me\noi
Since $\zeta=(\zeta_1,\ldots,\zeta_n)$ can be chosen small but otherwise arbitrary, we conclude that
$$\overline{A^t}\cdot\La-(n+1)\La=0 \, ,$$
which proves the theorem because $A^t=A$.
\end{proof}

In order to facilitate the use of equation (\ref{relation}) to estimate the order of $\Fa$, we use linear
invariance to assume that $\La=(\la, 0,\ldots,0)$ with $\la>0$. This normalization has a decoupling effect
on (\ref{relation}), with the matrix $A$ now given by
\begin{equation}\label{matrix-1} A_{ij}=S_{ij}^1\la-(n+1)S_{ij}^0+\frac{\delta_i^1\delta_j^1}{n+1}\,\la^2 \, ,\end{equation}
where
$S_{ij}^k=S_{ij}^kF_0(0)$. By equating the first components of (\ref{relation}) we obtain
\begin{equation}\label{eq-1}  \la^2+(n+1)S_{11}^1\la-(n+1)^2S_{11}^0-(n+1)^2=0 \, ,\end{equation}
while the remaining components give
\begin{equation}\label{eq-1} S_{1j}^1\la-(n+1)S_{1j}^0=0 \; \, , \; \, j=2,\ldots, n \, . \end{equation}

We are now in position to estimate the order of the family $\Fa$.

\begin{thm} The order of $\Fa$ satisfies
\begin{equation}\label{order}
\ord \leq \frac12(n+1)\left[\frac12\sqrt{n+1} \alpha+\sqrt{1+\frac14(n+1)\alpha^2+C(n,\alpha)}\;\right] \, ,
\end{equation}
where
$$ C(n,\alpha) \leq 6n^2\alpha^2+16\sqrt{n}\alpha \, . $$
\end{thm}

\begin{proof} From (\ref{eq-1}), we see that
$$\left(\la+\frac12(n+1)S_{11}^1\right)^2=(n+1)^2\left(1+\frac14(S_{11}^1)^2+S_{11}^0\right) \, .$$
In \cite{RH1}, the following bounds were established for the quantities $S_{11}^1, S_{11}^0$:
$$|S_{11}^1|\leq \sqrt{n+1} \alpha  \quad , \quad |S_{11}^0|\leq C(n,\alpha)\, ,$$ where
$$ C(n,\alpha)= \left(4n^2+2n-2+\frac{n+1}{n-1}\right)\alpha^2+
\left(4\sqrt{n+1}+8\frac{\sqrt{n+1}}{n-1}\right)\alpha\, .$$
The inequality (\ref{order}) follows at once from the estimates on $S_{11}^1, S_{11}^0$. Finally,
it is not difficult to see that
$$ C(n,\alpha) \leq 6n^2\alpha^2+16\sqrt{n}\alpha \, . $$
\end{proof}

The following corollary is obtained at once from (\ref{ordenes}).

\begin{cor} For the family $\Fa$ we have
$$ ||ord||\Fa \, \leq \, (n+1)\alpha+\sqrt{1+\frac14(n+1)\alpha^2+C(n,\alpha)}\, .$$

\end{cor}

\bibliographystyle{plain}


\bi
\noi
{\small Facultad de Matem\'aticas, Pontificia Universidad Cat\'olica de Chile,
Casilla 306, Santiago 22, Chile,\, \email{mchuaqui@mat.puc.cl}

\me
\noi
Facultad de Ciencias y Tecnolog\'ia,
Universidad Adolfo Ib\'a\~nez,
Av. Diagonal las Torres 2640, Pe\~nalolen, Chile,\,
\email{rodrigo.hernandez@uai.cl}

\end{document}